\documentclass[a4paper,10pt]{article}
\usepackage{latexsym,amsmath,amsthm,amssymb}
\usepackage{a4wide}
\usepackage{hyperref}
\usepackage{marginnote}
\usepackage{color}
\usepackage{cite}
\hypersetup{
pdftitle={improved MT in Hn}   
pdfauthor={Van Hoang Nguyen},
colorlinks = true,
linkcolor = magenta,
citecolor = blue,
}

\theoremstyle{plain}
\newtheorem{theorem}{Theorem}[section]

\newtheorem{proposition}[theorem]{Proposition}
\newtheorem{lemma}[theorem]{Lemma}

\theoremstyle{definition}

\theoremstyle{remark}

\renewcommand{\thefootnote}{\arabic{footnote}}

\def\R{\mathbb R}

\def\al{\alpha}
\def\om{\omega}
\def\Om{\Omega}
\def\be{\beta}

\def\si{\sigma}
\def\lam{\lambda}
\def\na{\nabla}
\def\la{\langle} 
\def\ra{\rangle} 
\def\lt{\left}
\def\rt{\right}

\def\i0i{\int_0^\infty}

\def\Vol{\text{Vol}}

\def\B{\mathbb B}
\def\H{\mathbb H}

\numberwithin{equation}{section}

\title{Improved Moser--Trudinger type inequalities in the hyperbolic space $\mathbb H^n$}
\author{Van Hoang Nguyen
\footnote{Institute of Research and Development, Duy Tan University, Da Nang, Vietnam}
}

\begin{document}
\maketitle



\renewcommand{\thefootnote}{}

\footnote{Email: 
\href{mail to: Van Hoang Nguyen <vanhoang0610@yahoo.com>}{vanhoang0610@yahoo.com}}

\footnote{2010 \emph{Mathematics Subject Classification\text}: 26D10, 46E35}

\footnote{\emph{Key words and phrases\text}: Moser--Trudinger inequality, exact growth, hyperbolic space, rearrangement, sharp constant}

\renewcommand{\thefootnote}{\arabic{footnote}}
\setcounter{footnote}{0}

\begin{abstract}
We establish an improved version of the Moser--Trudinger inequality in the hyperbolic space $\H^n$, $n\geq 2$. Namely, we prove the following result: for any $0 \leq \lam < \lt(\frac{n-1}n\rt)^n$, then we have
$$
\sup_{\substack{u\in C_0^\infty(\H^n)\\ \int_{\H^n} |\na_g u|_g^n d\Vol_g -\lam \int_{\H^n} |u|^n d\Vol_g \leq 1}} \int_{\H^n} \Phi_n(\al_n |u|^{\frac{n}{n-1}}) d\Vol_g < \infty,
$$
where $\al_n = n \om_{n-1}^{\frac1{n-1}}$, $\om_{n-1}$ denotes the surface area of the unit sphere in $\R^n$ and $\Phi_n(t) = e^t -\sum_{j=0}^{n-2}\frac{t^j}{j!}$. This improves the Moser--Trudinger inequality in hyperbolic spaces obtained recently by Mancini and Sandeep \cite{MS2010}, by Mancini, Sandeep and Tintarev \cite{MST2013} and by Adimurthi and Tintarev \cite{AT} for $\lambda = 0$. In the limiting case $\lam =(\frac{n-1}n)^n$, we prove a Moser--Trudinger inequality with exact growth in $\H^n$,
$$
\sup_{\substack{u\in C_0^\infty(\H^n)\\ \int_{\H^n} |\na_g u|_g^n d\Vol_g -(\frac{n-1}n)^n \int_{\H^n} |u|^n d\Vol_g \leq 1}} \frac{1}{\int_{\H^n} |u|^n d\Vol_g}\int_{\H^n} \frac{\Phi_n(\al_n |u|^{\frac{n}{n-1}})}{(1+ |u|)^{\frac n{n-1}}} d\Vol_g < \infty.
$$
This improves the Moser--Trudinger inequality with exact growth in $\H^n$ established by Lu and Tang \cite{LuTang}. These inequalities are achieved from the comparison of the symmetric non-increasing rearrangement of a function both in the hyperbolic and the Euclidean spaces, and the same inequalities in the Euclidean space. This approach seems to be new comparing with the previous ones.
\end{abstract}

\section{Introduction}
The classical Moser--Trudinger inequality states that for any bounded domain $\Om \subset \R^n$ then
\begin{equation}\label{eq:classicalMT}
\sup_{u \in W^{1,n}_0(\Om), \|\na u\|_n \leq 1} \int_\Om e^{\al_n |u|^{\frac n{n-1}}} dx < \infty,
\end{equation}
where $\al_n = n\om_{n-1}^{\frac1{n-1}}$, $\om_{n-1}$ denotes the surfaces area of the unit sphere in $\R^n$. The inequality \eqref{eq:classicalMT} is sharp in the sense that the supremum will becomes infinity if $\al_n$ is replaced by any constant $\al > \al_n$. It appears as the limiting case of the Sobolev embedding of $W_0^{1,n}(\Om)$ and was proved independently by Yudovi${\rm \check{c}}$ \cite{Y1961}, Poho${\rm \check{z}}$aev \cite{P1965} and Trudinger \cite{T1967}. The sharp form \eqref{eq:classicalMT} and the optimal constant $\al_n$ was found out by Moser \cite{M1970}. In the same work \cite{M1970}, Moser also proved the analogous sharp inequality on the Euclidean sphere with the aim of studying the problem of prescribing the Gaussian curvature on the sphere. The Moser--Trudinger has been generalized to higher order Sobolev spaces by Adams \cite{Adams1988} (nowaday, called Adams inequality) and to Riemannian manifolds \cite{CohnLu3,Li2001,Li2005,Yang,YSK} and sub-Riemannian manifolds \cite{CohnLu,CohnLu2,BMT,LamLu2013,LamTang,Yang12}. 

Another interesting and important question concerning to the Moser--Trudinger inequality \eqref{eq:classicalMT} is whether or not its extremal functions exist. The existence of extremal functions for \eqref{eq:classicalMT} was first proved by Carleson and Chang \cite{CC1986} when $\Om$ is the unit ball in $\R^n$, by Struwe \cite{Struwe} when $\Om$ is close to the unit ball in the sense of measure, by Flucher \cite{Flucher1992} and Lin \cite{Lin1996} when $\Om$ is a general smooth bounded domain, and by Li \cite{Li2001} for compact Riemannian surfaces. See also \cite{deFi} for futher existence results.


The classical Moser--Trudinger inequality \eqref{eq:classicalMT} is strengthened by Tintarev \cite{Tintarev} in the following way. Let $\lam_1(\Om)$ denotes the first non-zero eigenvalue of the Laplace operator in $H^1_0(\Om)$ with $\Om$ being smooth bounded domain in $\R^2$, i.e.,
\[
\lam_1(\Om) = \inf_{u\in H^1_0(\Om) \setminus\{0\}} \frac{\|\na u\|_2^2}{\|u\|_2^2},
\]
then for any $0 \leq \al < \lam_1(\Om)$, the quatity $\|u\|_{1,\al} = \lt(\|\na u\|_2^2 -\al \|u\|_2^2\rt)^{\frac12}$ defines a new norm on $H^1_0(\Om)$ which is equivalent to the Dirichlet norm $\|\na u\|_2$. In \cite{Tintarev}, the following improvement of the classical Moser--Trudinger inequality in dimension $2$ has been established
\begin{equation}\label{eq:Tintarev}
\sup_{u\in H_0^1(\Om), \|u\|_{1,\al} \leq 1} \int_\Om e^{4\pi u^2} dx < \infty.
\end{equation}
In \cite{Yang1}, based on the blow-up analysis method, Yang proved the existence of extremal functions for \eqref{eq:Tintarev} for any $0 \leq \al < \lam_1(\Om)$.

The Moser--Trudinger inequality \eqref{eq:classicalMT} has been extended to unbounded domains $\Om$ of $\R^n$ by Ruf \cite{Ruf} for $n=2$,by Li and Ruf \cite{LiRuf2008} for $n\geq 3$ and by Adimurthi and Yang \cite{AY} for a singular Moser--Trudinger inequality, i.e., they proved that for any $\tau >0$ and $\beta \in [0,n)$
\begin{equation}\label{eq:AY}
C_n(\tau,\beta):= \sup_{u\in W^{1,n}(\R^n): \|\na u\|_n^n + \tau \|u\|_n^n \leq 1} \int_{\R^n} \Phi_n(\al |u|^{\frac n{n-1}}) |x|^{-\beta} dx < \infty
\end{equation}
if and only if $\al \leq \al_n (1 -\frac\be n)$. The existence of extremal functions for \eqref{eq:AY} was proved by Ruf and Li and Ruf \cite{Ruf,LiRuf2008} for the case $\beta=0$ and by Li and Yang \cite{LiYang} for $\be >0$. We refer the reader to the paper of Lam and Lu \cite{LamLu} for another proof of \eqref{eq:AY} without using the rearrangement argument.


In this paper, we dicuss the Moser--Trudinger type inequality on the hyperbolic spaces $\H^n$ that is Riemannian manifolds of sectional curvature $-1$. In the following, we will use the Poincar\'e ball model for the hyperbolic space $\H^n$, i.e., the unit ball $\B^n$ in $\R^n$ equipped with the metric
\[
g(x) = \frac4{(1-|x|^2)^2} \sum_{i=1}^n dx_i^2.
\]
Let $\Vol_g$, $\na_g$ and $|\cdot|_g$ denote the volume element, gradient and scalar product with respect to the metric $g$. For simplicity, we shall use the notation $\|\na_g u\|_{n,g} = \lt(\int_{\B^n} |\na_g u|_g^n d\Vol_g\rt)^{\frac1n}$ and $\|u\|_{n,g} = \lt(\int_{\B^n} |u|^n d\Vol_g\rt)^{\frac1n}$ for any function $u$ on $\H^n$. The Moser--Trudinger inequality in the hyperbolic plane (i.e., $n=2$) has been established  by Mancini and Sandeep \cite{MS2010},
\begin{equation}\label{eq:MSH2}
\sup_{u \in W^{1,2}(\H^2), \|\na_g u\|_{2,g}^2 \leq 1} \int_{\B^2} \lt(e^{4\pi u^2} -1\rt) d\Vol_g < \infty.
\end{equation}
Another proof using the conformal group has been given by Adimurthi and Tintarev \cite{AT}. This idea has been extended to higher dimensions in \cite{Bat,Mancini}. A simple approach to the Moser--Trudinger inequality in the hyperbolic $\H^n$, $n \geq 3$ was given by Mancini, Sandeep and Tintarev \cite{MST2013} based on the radial estimates and decreasing rearrangement arguments. The inequality states that
\begin{equation}\label{eq:MSTHn}
\sup_{u \in W^{1,n}(\H^n), \|\na_g u\|_{n,g}^n \leq 1} \int_{\B^n} \Phi_n(\al_n |u|^{\frac n{n-1}}) d\Vol_g < \infty.
\end{equation}
See also \cite{WY2012,Nguyen2017,Tintarev} for the versions of the Moser--Trudinger inequality with a remainder term related to the metric of the Poincar\'e ball. We refer the reader to \cite{LuYang,NgoNguyen,KS2016} for the higher order extensions of the Moser--Trudinger inequality (i.e., Adams inequality) in hyperbolic space.

The main aim of this paper is to established the improved Moser--Trudinger inequality in hyperbolic space $\H^n$. Our proof also give another proof of the Moser--Trudinger inequality \eqref{eq:MSTHn}. To state our main result, let us recall the Hardy inequality (or Poincar\'e--Sobolev inequality) in $\H^n$
\begin{equation}\label{eq:Hardy}
\int_{\B^n} |\na_g u|_g^n d\Vol_g \geq \lt(\frac{n-1}n\rt)^n \int_{\B^n} |u|^n d\Vol_g,\qquad \forall\, u\in W^{1,n}(\H^n).
\end{equation}
The constant $(\frac{n-1}n)^n$ is sharp and never achieved. This inequality was proved by Mancini and Sandeep \cite{MS08} for $n =2$ and by Mancini, Sandeep and Tintarev \cite{MST2013} for any $n\geq 3$ (see \cite{NgoNguyen1} for more general results). Our first main result of this paper reads as follows

\begin{theorem}\label{Maintheorem}
Let $n\geq 2$. For any $0 \leq \lam < (\frac{n-1}n)^n$, it holds
\begin{equation}\label{eq:Maininequality}
\sup_{u\in W^{1,n}(\H^n), \|\na_g u\|_{n,g}^n -\lam \|u\|_{n,g}^n  \leq 1} \int_{\B^n} \Phi_n(\al_n |u|^{\frac n{n-1}}) d\Vol_g < \infty.
\end{equation}
\end{theorem}
Note that Theorem \ref{Maintheorem} contains the Moser--Trudinger inequality \eqref{eq:MSTHn} as a special case corresponding to $\lam =0$. Moreover, it gives an improvement version of the Moser--Trudinger inequality \eqref{eq:MSTHn} in the sense of Tintarev's improvement for the classical Moser--Trudinger inequality (see the inequality \eqref{eq:Tintarev}). Evidently, the constant $\al_n$ in \eqref{eq:Maininequality} is optimal because of the sharpness of the Moser--Trudinger inequality \eqref{eq:MSTHn}. In the limiting case $\lam =(\frac{n-1}n)^n$, the supremum in \eqref{eq:Maininequality} will be infinity (see Conjecture $5.2$ in \cite{MST2013}). In this case, a Hardy--Moser--Trudinger inequality is established in \cite{Nguyen2017} which generalizes the result of Wang and Ye \cite{WY2012} in dimension $2$ to any dimensions (see also \cite{YangZhu} for an improved Hardy--Trudinger--Moser inequality in two dimensional hyperbolic space which improves the inequality of Wang and Ye). Our next aim of this paper is to establish, in this limit case, another kind of Moser--Trudinger inequality (the so-called Moser--Trudinger inequality with exact growth) in the hyperbolic space as follows.

\begin{theorem}\label{Exact}
Let $n \geq 2$. Then the following inequality holds
\begin{equation}\label{eq:MTexact}
\sup_{u\in W^{1,n}(\H^n), \|\na_g u\|_{n,g}^n -(\frac{n-1}n)^n \|u\|_{n,g}^n \leq 1} \frac1{\|u\|_{n,g}^n}\int_{\B^n} \frac{\Phi_n(\al_n |u|^{\frac n{n-1}})}{(1+ |u|)^{\frac n{n-1}}} d\Vol_g < \infty.
\end{equation}
The inequality \eqref{eq:MTexact} is sharp in the sense that if we replace $\al_n$ by any constant $\al > \al_n$ or the power $\frac n{n-1}$ in the denominator by any $p < \frac n{n-1}$ then the supremum will be infinity.
\end{theorem}
The Moser--Trudinger inequality with exact growth in Euclidean space $\R^n$ was proved by Ibrahim, Masmoudi and Nakanishi \cite{IMN} in the plane (i.e, $n=2$) and by Masmoudi and Sani \cite{MSani} for $n\geq 3$. This inequality states that
\begin{equation}\label{eq:MTExactRn}
\sup_{u\in H^{1,n}(\R^n), \|\na u\|_n \leq 1} \frac{1}{\|u\|_n^n} \int_{\R^n} \frac{\Phi_n(\al_n |u|^{\frac{n}{n-1}})}{(1 + |u|)^{\frac n{n-1}}} dx < \infty.
\end{equation}
It was also shown in \cite{MSani} that the inequality \eqref{eq:MTExactRn} is sharp in the sense that if we replace $\al_n$ by any constant $\al > \al_n$ or the power $\frac n{n-1}$ in the denominator by any $p < \frac n{n-1}$ then the supremum will be infinity. This kind of inequality was extended to the hyperbolic spaces by Lu and Tang \cite{LuTang} in the form
\begin{equation}\label{eq:LuTang}
\sup_{u\in W^{1,n}(\H^n), \|\na_g u\|_{n,g} \leq 1} \frac{1}{\|u\|_{n,g}^n} \int_{\B^n} \frac{\Phi_n(\al_n |u|^{\frac n{n-1}})}{(1+ |u|)^{\frac{n}{n-1}}} d\Vol_g < \infty.
\end{equation}
Again, the power $\frac n{n-1}$ in the denominator is sharp. Comparing with the inequality \eqref{eq:LuTang}, our inequality \eqref{eq:MTexact} is stronger than the one of Lu and Tang. We refer the reader to \cite{MSani14,MSani17,LuTangZhu,NgoNguyen,Kamarkar} for the Adams inequality with exact growth both in the Euclidean and hyperbolic spaces.

Let us explain the idea in the proof of Theorem \ref{Maintheorem} and Theorem \ref{Exact}. For any function $u\in W^{1,n}(\H^n)$ we define a function $u^*$ which is non-increasing rearrangement function of $u$ (see the precise definition in Section $2.2$ below). From this function $u^*$ we define two new functions $u^\sharp_g$ on $\H^n$ and $u^\sharp_e$ on $\R^n$ by
\[
u^\sharp_g(x) = u^*(\Vol_g(B_g(0,\rho(x)))),\quad x\in \B^n,
\]
where $ \rho(x) = \ln \frac{1+|x|}{1-|x|}$ denotes the geodesic distance from $x$ to $0$, and $B_g(0,r)$ denotes the open geodesic ball center at $0$ and radius $r>0$ in $\H^n$, and
\[
u^\sharp_e(x) = u^*(\si_n |x|^n), \quad x\in \R^n
\]
where $\si_n$ denotes the volume of unit ball in $\R^n$, respectively. It is remarkable that $u^\sharp_g$ and $u^\sharp_e$ has the same non-increasing rearrangement function (which is $u^*$).
Our main ingredient in the proof of Theorem \ref{Maintheorem} is the relation between $\int_{\B^n} |\na_g u^\sharp_g|^n d\Vol_g$ and $\int_{\R^n} |\na u^\sharp_e|^n dx$ which is the content of the following theorem

\begin{theorem}\label{Maintheorem1}
Let $n \geq 2$. It holds
\begin{equation}\label{eq:mainrelation}
\int_{\B^n} |\na_g u^\sharp_g|^n d\Vol_g - \lt(\frac{n-1}n\rt)^n \int_{\B^n} |u_g^\sharp|^n d\Vol_g \geq \int_{\R^n} |\na u^\sharp_e|^n dx.
\end{equation}
\end{theorem}
Theorem \ref{Maintheorem1} combining with \eqref{eq:AY} and \eqref{eq:MTExactRn} would imply Theorem \ref{Maintheorem} and Theorem \ref{Exact}, respectively. It is worthy to emphasize here that Theorem \ref{Maintheorem1} also implies the Hardy inequality (or Poincar\'e--Sobolev inequality) \eqref{eq:Hardy} by the P\'olya--Szeg\"o principle.

The rest of this paper is organized as follows. In Section $2$, we recall some facts of the hyperbolic spaces and the symmetrization of the functions in the hyperbolic spaces. In Section $3$ we give the proof of our main results, i.e., the proof of Theorem \ref{Maintheorem}, Theorem \ref{Exact} and Theorem \ref{Maintheorem1}.

\section{Preliminaries}
\subsection{Background on hyperbolic spaces}
The hyperbolic space $\H^n$, $n\geq 2$ is a complete, simply connected Riemannian manifold having constant sectional curvature equal to $-1$, and for a given dimensional number, any two such spaces are isometries \cite{Wol67}. There is a number of models for $\H^n$, however, the most important models are the half-space model, the Poincar\'e ball model, and the hyperboloid or Lorentz model. In this paper, we will use the Poincar\'e ball model since this model is especially useful for questions involving rotational symmetry. Given $n\geq 2$, we denote by $\B^n$ the open unit ball in $\R^n$ centered at origin. The Poincar\'e ball model of the hyperbolic space $\H^n$ is the unit ball $\B^n$ equipped with the metric
\[
g(x) = \frac4{(1-|x|^2)^2} \sum_{i=1}^n dx_i^2.
\]
The volume element with respect to Riemannian metric is 
\[
d\Vol_g = \frac{2^n}{(1 -|x|^2)^n} dx.
\]
For $x \in \B^n$, denote $\rho(x) =d(x,0) = \ln \frac{1+|x|}{1-|x|}$ the geodesic distance from $x$ to $0$, and for $r >0$ denote $B_g(0,r)$ the open geodesic ball center at $0$ and radius $r$. We still use $\na$ to denote the Euclidean gradient in $\R^n$ as well as $\la\cdot, \cdot\ra$ to denote the standard inner product in $\R^n$. With respect to the metric $g$, the hyperbolic gradient $\na_g$ and the inner product $\la \cdot,\cdot\ra_g$ in each tangent space of $\H^n$ are given by
\[
\na_g = \frac{(1-|x|^2)^2} 4 \na,\qquad \la\cdot,\cdot\ra_g = \frac4{(1-|x|^2)^2} \la\cdot,\cdot\ra,
\]
respectively. For simplicity, we shall use the notation $|\na_g u|_g = \sqrt{\la \na_g u,\na_g u\ra_g}$ for a smooth function $u$ in $\H^n$. With these notation, we have the relation
\begin{equation}\label{eq:relnorm}
\int_{\B^n} |\na_g u|^n_g d\Vol_g = \int_{\B^n} |\na u|^n dx.
\end{equation}
We associate with this form a Sobolev space which is the completion of $C_0^\infty(\B^n)$ with respect the form \eqref{eq:relnorm} above. This space which will be denoted by $W^{1,n}(\B^n)$ is identified on the Poincar\'e ball model as the standard Sobolev space $W^{1,n}_0(\B^n)$ equipped with the norm $\lt(\int_{\B^n} |\na u|^n dx\rt)^{\frac1n}$. By $W^{1,n}_{0,r}(\B^n)$ we denote the subspace of radially symmetric functions of $W^{1,n}_0(\B^n)$.

\subsection{Symmetric decreasing rearrangements}
It is now known that the symmetrization argument works well in the setting of the hyperbolic spaces $\H^n$. Let us recall some facts about the rearrangement in the hyperbolic spaces. Let $u: \H^n \to \R$ be a function such that
\[
\Vol_g(\{x\in \H^n\,:\, |u(x)|> t\}) = \int_{\{x\in \H^n\,:\, |u(x)|> t\}} d\Vol_g < \infty,\quad \forall\, t>0.
\]
For such a function $u$, its distribution function, denoted by $\mu_u$, is defined by
\[
\mu_u(t) = \Vol_g\{x\in \H^n\, :\, |u(x)| > t\}, \qquad t >0.
\]
The function $(0,\infty)\ni t\mapsto \mu_u(t)$ is non-increasing and right-continuous. Then the non-increasing rearrangement function $u^*$ of $u$ is defined by
\[
u^*(t) = \sup\{s >0\, :\, \mu_u(s) > t\}.
\]
Note that the function $(0,\infty) \ni t \to u^*(t)$ is non-increasing. We now define the symmetric non-increasing rearrangement function $u_g^\sharp$ of $u$ by
\begin{equation}\label{eq:usharpg}
u^\sharp_g(x) = u^*(\Vol_g(B_g(0,\rho(x)))),\quad x \in \B^n.
\end{equation}
We also define a function $u^\sharp_e$ on $\R^n$ by
\begin{equation}\label{eq:usharpe}
u^\sharp_e(x) = u^*(\si_n |x|^n),\quad x\in \R^n,
\end{equation}
where $\si_n$ denotes the volume of unit ball in $\R^n$. Since $u$, $u_g^\sharp$ and $u^\sharp_e$ has the same non-increasing rearrangement function (which is $u^*$), then we have
\begin{equation}\label{eq:equalkey}
\int_{\B^n} \Phi(|u|) d\Vol_g = \int_{\B^n} \Phi(u^\sharp_g) d\Vol_g = \int_{\R^n} \Phi(u_e^\sharp) dx = \int_0^\infty \Phi(u^*(t)) dt,
\end{equation}
for any increasing function $\Phi: [0,\infty) \to [0,\infty)$ with $\Phi(0) =0$. This equality is a consequence of layer cake representation. Moreover, by P\'olya--Szeg\"o principle, we have
\begin{equation}\label{eq:PSprinciple}
\int_{\B^n} |\na_g u^\sharp_g|_g^n d\Vol_g \leq \int_{\B^n} |\na_g u|_g^n d\Vol_g.
\end{equation}
We finish this section by recall the polar coordinate formula on $\H^n$: for any function $f:[0,\infty) \to \R$, then it holds
\begin{equation}\label{eq:polar}
\int_{\B^n} f(\rho(x)) d\Vol_g = \om_{n-1} \int_0^\infty f(t) (\sinh t)^{n-1} dt.
\end{equation}


\section{Proof of the main results}
\subsection{Proof of Theorem \ref{Maintheorem1}}
In this section, we give the proof of Theorem \ref{Maintheorem1}. Indeed, we will obtain a stronger result as we will see below. Given a function $u \in W^{1,n}(\H^n)$. We can assume that $u\not\equiv 0$, if not there is nothing to do. Let $u_g^\sharp$ and $u^\sharp_e$ be defined by \eqref{eq:usharpg} and \eqref{eq:usharpe}, respectively. For simplicity, we shall use the notation $v = u^*$.

It is easy to see that
\begin{equation}\label{eq:euclideangra}
\int_{\R^n} |\na u^\sharp_e|^n dx = (n\si_n)^n \int_0^\infty |v'(s)|^n \lt(\frac{s}{\si_n}\rt)^{n-1} ds.
\end{equation}
It is remarked that
\[
\Vol_g(B_g(0,\rho(x))) = n\si_n \int_0^{\rho(x)} (\sinh t)^{n-1} dt,
\]
hence
\begin{equation}\label{eq:gradientvolume}
\na_g \Vol_g(B_g(0,\rho(x))) = n\si_n (\sinh \rho(x))^{n-1} \na_g \rho(x).
\end{equation}
Since $|\na_g \rho(x)|_g =1$, we then have by using polar coordinate formula \eqref{eq:polar}
\begin{align}\label{eq:gra1}
\int_{\B^n} |\na_g u^{\sharp}_g|_g^n d\Vol_g&= (n\si_n)^n \int_{\B^n} |v'(\Vol_g(B_g(0,\rho(x))))|^n (\sinh \rho(x))^{n(n-1)} d\Vol_g\notag\\
&= (n\si_n)^{n+1} \int_0^\infty |v'(\Vol_g(B_g(0,r)))|^n (\sinh r)^{n(n-1)} dr.
\end{align}
Let us define the function
\[
\Phi(r) = n \int_0^r (\sinh t)^{n-1} dt = \frac1{\si_n} \Vol_g(B_g(0,r)).
\]
The function $\Phi: [0,\infty) \to [0,\infty)$ is diffeomorphism and strictly increasing. Making the change of variable $s = \si_n\Phi(r)$ in \eqref{eq:gra1}, we obtain
\begin{equation}\label{eq:gra2}
\int_{\B^n} |\na_g u^{\sharp}_g u|_g^n d\Vol_g = (n\si_n)^n \int_0^\infty |v'(s)|^n \lt(\sinh \Phi^{-1}\lt(\frac s{\si_n}\rt)\rt)^{n(n-1)} ds.
\end{equation}
Let us define
\[
k(s) = (\sinh \Phi^{-1}(s))^{n(n-1)} - s^{n-1},\quad s \geq 0.
\]
We then get from \eqref{eq:euclideangra} and \eqref{eq:gra2} that
\begin{equation}\label{eq:gra3}
\int_{\B^n} |\na_g u^{\sharp}_g u|_g^n d\Vol_g = \int_{\R^n} |\na u^\sharp_e|^n dx + (n \si_n)^n \int_0^\infty |v'(s)|^n k\lt(\frac{s}{\si_n}\rt) ds.
\end{equation}
In order to prove Theorem \ref{Maintheorem1}, we need an estimate from below for the function $k$. In fact, we have the following result
\begin{lemma}\label{lemmakey}
Let $n \geq 2$. It holds
\begin{equation}\label{eq:boundk}
k(\Phi(t)) \geq \lt(\frac{n-1}n\rt)^n \Phi(t)^n, \quad t \geq 0.
\end{equation}
\end{lemma}
\begin{proof}
Note that $k(\Phi(t)) = (\sinh t)^{n(n-1)} - \Phi(t)^{n-1}$. So, \eqref{eq:boundk} is equivalent to
\begin{equation}\label{eq:equivform}
F(t) = (\sinh t)^{n(n-1)} - \Phi(t)^{n-1} - \lt(\frac{n-1}n\rt)^n \Phi(t)^n \geq 0,\quad t \geq 0.
\end{equation}

If $n =2$, then we easily compute that $\Phi(t) = 2(\cosh t -1)$. Therefore, we have
\[
F(t) = (\sinh t)^2 -2(\cosh t-1) -(\cosh t-1)^2=0,
\]
here we use the equality $(\cosh t)^2 - (\sinh t)^2 =1$. Thus, \eqref{eq:boundk} is an equality if $n=2$.

Suppose $n\geq 3$, differentiating the function $F$, we get
\begin{align*}
F'(t) &= n(n-1) \lt(\sinh t\rt)^{n(n-1) -1} \cosh t - n(n-1) (\sinh t)^{n-1} \Phi(t)^{n-2} \\
&\qquad - \lt(\frac{n-1}n\rt)^n n^2 (\sinh t)^{n-1} \Phi(t)^{n-1}\\
&=n(n-1) (\sinh t)^{n-1}\lt(\lt(\sinh t\rt)^{n(n-2)} \cosh t -\Phi(t)^{n-2} -\lt(\frac{n-1}n\rt)^{n-1} \Phi(t)^{n-1}\rt)\\
&=: n(n-1) (\sinh t)^{n-1} G(t).
\end{align*}
We next differentiate the function $G$ to obtain
\begin{align*}
G'(t) &= n(n-2)\lt(\sinh t\rt)^{n(n-2)-1} (\cosh t)^2 + \lt(\sinh t\rt)^{(n-1)^2} -n(n-2) (\sinh t)^{n-1} \Phi(t)^{n-3} \\
&\quad - \lt(\frac{n-1}n\rt)^{n-1} n(n-1) (\sinh t)^{n-1} \Phi(t)^{n-2}.
\end{align*}
Using the equality $(\cosh t)^2 = 1+ (\sinh t)^2$, we simplify the expression of $G'$ by
\begin{align*}
G'(t) &=  (n-1)^2 \lt(\sinh t\rt)^{(n-1)^2} + n(n-2)\lt(\sinh t\rt)^{n(n-2)-1} -n(n-2) (\sinh t)^{n-1} \Phi(t)^{n-3} \\
&\quad - \lt(\frac{n-1}n\rt)^{n-1} n(n-1) (\sinh t)^{n-1} \Phi(t)^{n-2}\\
&=(n-1)^2 (\sinh t)^{n-1}\Bigg[(\sinh t)^{(n-1)(n-2)} + \frac{n(n-2)}{(n-1)^2}\lt((\sinh t)^{n(n-3)} -\Phi(t)^{n-3}\rt)\\
&\hspace{6cm} - \lt(\frac{n-1}n\rt)^{n-2}\Phi(t)^{n-2}\Bigg]\\
&=(n-1)^2 (\sinh t)^{n-1} H(t).
\end{align*}
Easy estimates show that
\[
\Phi(t) = n \int_0^t (\sinh s)^{n-1} ds < n \int_0^t (\sinh s)^{n-1} \cosh s ds= (\sinh t)^n,\qquad t >0,
\]
and
\[
\Phi(t) = n \int_0^t (\sinh s)^{n-1} ds < n \int_0^t (\sinh s)^{n-2} \cosh s ds = \frac n{n-1} (\sinh t)^{n-1},\qquad t >0.
\]
Using these estimates for $\Phi$, we get $H(t) >0$ for $t>0$ which is equivalent to $G'(t) >0$ for $t>0$. Consequently, $G(t) > G(0) =0$ for any $t>0$ which is equivalent to $F'(t) >0$ for any $t >0$. Hence, we get $F(t) > F(0) =0$ for any $t>0$. This finishes our proof.
\end{proof}

It follows from \eqref{eq:gra3} and Lemma \ref{lemmakey} that
\begin{equation}\label{eq:equalnormn}
\int_{\H^n} |\na_g u^\sharp_g(x)|^n d\Vol_g \geq \int_{\R^n} |\na u^\sharp_e|^n dx + (n-1)^n \int_0^\infty |v'(s)|^n s^n ds.
\end{equation}
We next make a change of function $w(s) = v(s) s^{\frac1n}$ or equivalently $v(s) = w(s) s^{-\frac1n}$. Differentiating the function $v$ implies
\[
v'(s) = w'(s) s^{-\frac1n} - \frac1n w(s) s^{-\frac1n -1}.
\]
Note that $v'(s) \leq 0$ since $v =u^*$ is a non-increasing function. It then is easy to verify that
\[
|a -b|^n \geq |a|^n + |b|^n - n a b^{n-1}
\]
with $a -b \leq 0$ and $b \geq 0$. Hence
\begin{align*}
\int_0^\infty |v'(s)|^n s^n ds &\geq \int_0^\infty |w'(s)|^n s^{n-1} ds + \frac1{n^n}\int_0^\infty w(s)^n s^{-1} ds\\
&\qquad - \int_0^\infty w'(s) s^{-\frac1n} w(s)^{n-1}s^{-\frac{n^2-1}n} s^n ds\\
&= \int_0^\infty |w'(s)|^n s^{n-1} ds + \frac1{n^n}\int_0^\infty v(s)^n ds,
\end{align*}
the equality follows by intgeration by parts. Combining the previous estimates together with \eqref{eq:equalnormn}, we obtain
\begin{multline}\label{eq:onHn}
\int_{\H^n} |\na_g u^\sharp_g|^n d\Vol_g -\lt(\frac{n-1}n\rt)^n \int_{\H^n} |u_g^\sharp|^n d\Vol_g\\
\geq \int_{\R^n} |\na u^\sharp_e|^n dx + (n-1)^n \int_0^\infty |(v(s) s^{\frac1n})'|^n s^{n-1} ds.
\end{multline}
The estimate \eqref{eq:onHn} finishes the proof of Theorem \ref{Maintheorem1}. Indeed, it is even stronger than the statement of Theorem \ref{Maintheorem1}, and we believe that it could give a proof of the Poincar\'e--Sobolev inequality in \cite[Lemma $2.1$, part (b)]{MST2013}.

\subsection{Proof of Theorem \ref{Maintheorem}}
In this section, we give the proof of Theorem \ref{Maintheorem}. Suppose $0 \leq \lam < (\frac{n-1}n)^n$, denote
\[
\tau_\lam = \lt(\frac{n-1}n\rt)^n - \lam >0.
\]
Let $u \in W^{1,n}(\H^n)$ be a function with 
\[
\int_{\B^n} |\na_g u|_g^n d\Vol_g - \lam \int_{\B^n} |u|^n d\Vol_g \leq 1.
\]
We define two new functions $u_g^\sharp$ and $u^\sharp_e$ by \eqref{eq:usharpg} and \eqref{eq:usharpe} respectively. By P\'olya--Szeg\"o principle \eqref{eq:PSprinciple} and the equality \eqref{eq:equalkey}, we have
\[
\int_{\B^n} |\na_g u_g^\sharp|_g^n d\Vol_g - \lam \int_{\B^n} |u^\sharp_g|^n d\Vol_g \leq \int_{\B^n} |\na_g u|_g^n d\Vol_g - \lam \int_{\B^n} |u|^n d\Vol_g \leq 1.
\]
Theorem \ref{Maintheorem1} and the equality \eqref{eq:equalkey} imply
\[
\int_{\B^n} |\na_g u_g^\sharp|_g^n d\Vol_g - \lam \int_{\B^n} |u^\sharp_g|^n d\Vol_g \geq \int_{\R^n} |\na u^\sharp_e|^n dx + \tau_\lam \int_{\R^n} |u_e^\sharp|^n dx.
\]
Combining these two estimates together, we arrive
\[
\int_{\R^n} |\na u^\sharp_e|^n dx + \tau_\lam \int_{\R^n} |u_e^\sharp|^n dx \leq 1.
\]
As a consequence of the previous estimate, the Moser--Trudinger inequality \eqref{eq:AY} and the equality \eqref{eq:equalkey}, we get
\begin{align*}
\int_{\B^n} \Phi_n(\al_n |u|^{\frac n{n-1}}) d\Vol_g &= \int_{\R^n} \Phi_n(\al_n |u^\sharp_e|^{\frac n{n-1}}) dx\\
&\leq \sup_{v \in W^{1,n}(\R^n), \|\na v\|_n^n + \tau_\lam \|v\|_n^n \leq 1}\int_{\R^n} \Phi_n(\al_n |v|^{\frac n{n-1}}) dx\\
&=C_n(\tau_\lam,0).
\end{align*}
Taking the supremum over such functions $u$, we get
\[
\sup_{\substack{u\in W^{1,n}(\H^n)\\ \|\na_g u\|_{n,g}^n  - \lam\|u\|_{n,g}^n \leq 1}} \int_{\B^n} \Phi_n(\al_n |u|^{\frac n{n-1}}) d\Vol_g \leq C_n(\tau_\lam,0) < \infty.
\] 
This finishes the proof of Theorem \ref{Maintheorem}.

We conclude this subsection by giving a lower bound for
\[
\sup_{\substack{u\in W^{1,n}(\H^n), \|\na_g u\|_{n,g}^n  - \lam\|u\|_{n,g}^n \leq 1}} \int_{\B^n} \Phi_n(\al_n |u|^{\frac n{n-1}}) d\Vol_g.
\]
\begin{proposition}
Let $n \geq 2$. It holds
\begin{equation}\label{eq:lowerbound}
\sup_{\substack{u\in W^{1,n}(\H^n)\\ \|\na_g u\|_{n,g}^n  - \lam\|u\|_{n,g}^n \leq 1}} \int_{\B^n} \Phi_n(\al_n |u|^{\frac n{n-1}}) d\Vol_g \geq \frac{\al_n^{n-1}}{(n-1)!} \lt(\lt(\frac{n-1}n\rt)^n -\lam\rt)^{-1}.
\end{equation}
\end{proposition}
In the case $\lam =0$, the lower bound \eqref{eq:lowerbound} was proved in \cite[Lemma $4.1$]{MST2013}. The similar bound in the Euclidean space can be found in \cite{IW}.
\begin{proof}
We use the functions $\psi_k$ constructed in \cite{MST2013}. For $k \geq 1$, let us define the function $\psi_k$ by
\[
\psi_k(x) = (1 -|x|^2)^{\frac{n-1}n + \frac1{nk}}.
\]
We can easily check that
\[
\int_{\B^n} \psi_k^n d\Vol_g = \om_{n-1} 2^n \int_0^1 (1-r^2)^{-1 + \frac1k} r^{n-1} dr = \om_{n-1} 2^{n-1} B\lt(\frac1k, \frac n2\rt),
\]
where $B(a,b)$ denotes the usual beta function and 
\begin{align*}
\int_{\B^n} |\na_g \psi_k|_g^n d\Vol_g &= \om_{n-1}\lt(\frac{n-1}n + \frac1{nk}\rt)^n 2^n \int_0^1 (1-r^2)^{-1 + \frac1k} r^{2n-1} dr\\
&= \om_{n-1}2^{n-1}\lt(\frac{n-1}n + \frac1{nk}\rt)^n  B\lt(\frac1k, n\rt).
\end{align*}
We choose $a_k$ such that $\|\na_g(a_k \psi_k)\|_{n,g}^n -\lam \|a_k\psi_k\|_{n,g}^n =1$, i.e.,
\[
a_k^n = \om_{n-1}^{-1} 2^{1-n}\lt(\lt(\frac{n-1}n + \frac1{nk}\rt)^n  B\lt(\frac1k, n\rt) - \lam B\lt(\frac1k, \frac n2\rt)\rt)^{-1}.
\]
Hence
\[
a_k^n \int_{\B^n} |\psi_k|^n d\Vol_g = \lt(\lt(\frac{n-1}n + \frac1{nk}\rt)^n \frac{ B\lt(\frac1k, n\rt)}{B\lt(\frac1k, \frac n2\rt)} - \lam\rt)^{-1}.
\]
We can readily check that
\[
\lim_{k\to \infty} \frac{ B\lt(\frac1k, n\rt)}{B\lt(\frac1k, \frac n2\rt)} =1,
\]
which yields
\[
\lim_{k\to \infty} a_k^n \int_{\B^n} |\psi_k|^n d\Vol_g = \lt(\lt(\frac{n-1}n\rt)^n  - \lam\rt)^{-1}.
\]
By the simple estimate $\Phi_n(t) \geq \frac{t^{n-1}}{(n-1)!}$ for $t\geq 0$, we obtain
\begin{align*}
\sup_{\substack{u\in W^{1,n}(\H^n)\\ \|\na_g u\|_{n,g}^n  - \lam\|u\|_{n,g}^n \leq 1}} \int_{\B^n} \Phi_n(\al_n |u|^{\frac n{n-1}}) d\Vol_g&\geq \liminf_{k\to \infty}\int_{\B^n} \Phi_n(\al_n |a_k\psi_k|^{\frac n{n-1}}) d\Vol_g\\
&\geq \frac{\al_n^{n-1}}{(n-1)!} \lim_{k\to \infty} a_k^n \int_{\B^n} |\psi_k|^n d\Vol_g\\
&= \frac{\al_n^{n-1}}{(n-1)!} \lt(\lt(\frac{n-1}n\rt)^n -\lam\rt)^{-1}
\end{align*}
as wanted.
\end{proof}


\subsection{Proof of Theorem \ref{Exact}}
This section is devoted to prove Theorem \ref{Exact}. The proof follows the lines in the proof of Theorem \ref{Maintheorem}. Let $u \in W^{1,n}(\H^n)$ be a function with 
\[
\int_{\B^n} |\na_g u|_g^n d\Vol_g - \lt(\frac{n-1}n\rt)^n \int_{\B^n} |u|^n d\Vol_g \leq 1.
\]
We define two new functions $u_g^\sharp$ and $u^\sharp_e$ by \eqref{eq:usharpg} and \eqref{eq:usharpe} respectively. By P\'olya--Szeg\"o principle \eqref{eq:PSprinciple} and the equality \eqref{eq:equalkey}, we have
\[
\int_{\B^n} |\na_g u_g^\sharp|_g^n d\Vol_g - \lt(\frac{n-1}n\rt)^n \int_{\B^n} |u^\sharp_g|^n d\Vol_g \leq \int_{\B^n} |\na_g u|_g^n d\Vol_g - \lt(\frac{n-1}n\rt)^n \int_{\B^n} |u|^n d\Vol_g \leq 1.
\]
Theorem \ref{Maintheorem1} and the equality \eqref{eq:equalkey} imply
\[
\int_{\B^n} |\na_g u_g^\sharp|_g^n d\Vol_g - \lt(\frac{n-1}n\rt)^n \int_{\B^n} |u^\sharp_g|^n d\Vol_g \geq \int_{\R^n} |\na u^\sharp_e|^n dx.
\]
Combining these two estimates together, we arrive
\[
\int_{\R^n} |\na u^\sharp_e|^n dx \leq 1.
\]
As a consequence of the previous estimate and the equality \eqref{eq:equalkey}, we get
\begin{align*}
\frac1{\|u\|_{n,g}^n}\int_{\B^n} \frac{\Phi_n(\al_n |u|^{\frac n{n-1}})}{(1+ |u|)^{\frac{n}{n-1}}} d\Vol_g &= \frac1{\|u^\sharp_e\|_n^n}\int_{\R^n} \frac{\Phi_n(\al_n |u^\sharp_e|^{\frac n{n-1}})}{(1 + |u_e^\sharp|)^{\frac n{n-1}}} dx\\
&\leq \sup_{v \in W^{1,n}(\R^n), \|\na v\|_n \leq 1}\frac1{\|v\|_n^n}\int_{\R^n} \frac{\Phi_n(\al_n |v|^{\frac n{n-1}})}{(1 + |v|)^{\frac n{n-1}}} dx.
\end{align*}
Taking the supremum over such functions $u$, we get
\[
\sup_{\substack{u\in W^{1,n}(\H^n)\\ \|\na_g u\|_{n,g}^n - (\frac{n-1}n)^n \|u\|_{n,g}^n \leq 1}} \frac1{\|u\|_{n,g}^n}\int_{\B^n} \frac{\Phi_n(\al_n |u|^{\frac n{n-1}})}{(1+ |u|)^{\frac{n}{n-1}}} d\Vol_g \leq \sup_{\substack{v \in W^{1,n}(\R^n)\\ \|\na v\|_n \leq 1}}\frac1{\|v\|_n^n}\int_{\R^n} \frac{\Phi_n(\al_n |v|^{\frac n{n-1}})}{(1 + |v|)^{\frac n{n-1}}} dx.
\] 
The right hand side is finite by the Moser--Trudinger inequality with exact growth \eqref{eq:MTExactRn}. This proves \eqref{eq:MTexact}.

It remains to check the sharpness of \eqref{eq:MTexact}. In order to do this, we need construct a sequence $\{u_k\}$ in $W^{1,n}(\H^n)$ such that $\|\na_g u_k\|_{n,g}^n - (\frac{n-1}n)^n \|u_k\|_{n,g}^n =1$ and
\begin{equation}\label{eq:example}
\lim_{k\to \infty}\frac1{\|u_k\|_{n,g}^n}\int_{\B^n} \frac{\Phi_n(\al |u_k|^{\frac n{n-1}})}{(1+ |u_k|)^{p}} d\Vol_g =\infty.
\end{equation}
if $\al > \al_n$ or $p < \frac n{n-1}$. Let us define the sequence $\{u_k\}_k$ as follows
\[
u_k(x) = \om_{n-1}^{-\frac1n} C_k \times
\begin{cases}
k^{\frac{n-1}n} &\mbox{if $0\leq \rho(x) < e^{-k}$}\\
k^{\frac{n-1}n} \frac{-\ln \rho(x)} k&\mbox{if $e^{-k} \leq \rho(x) < 1$}\\
0&\mbox{if $\rho(x) \geq 1$,}
\end{cases}
\]
where $C_k$ is chosen such that
\[
\|\na_g u_k\|_{n,g}^n - \lt(\frac{n-1}n\rt)^n \|u_k\|_{n,g}^n =1.
\]
Therefore, a straightforward computation shows that
\begin{align*}
C_k &=\Bigg(\frac1k \int_{e^{-k}}^1 t^{-n} (\sinh t)^{n-1} dt - \lt(\frac{n-1}n\rt)^nk^{n-1} \int_0^{e^{-k}} (\sinh t)^{n-1}dt\\
&\hspace{5cm} -\lt(\frac{n-1}n\rt)^n k^{-1} \int_{e^{-k}}^1 (-\ln t)^n (\sinh t)^{n-1} dt\Bigg)^{-\frac1n}.
\end{align*}
It is easy to check that that
\[
k^{-1} \int_{e^{-k}}^1 (-\ln t)^n (\sinh t)^{n-1} dt = O(k^{-1}),
\]
\[
k^{n-1} \int_0^{e^{-k}} (\sinh t)^{n-1}dt = O(k^{n-1} e^{-nk}),
\]
and
\[
\frac1k \int_{e^{-k}}^1 t^{-n} (\sinh t)^{n-1} dt = 1+ \frac1k \int_{e^{-k}}^1\frac1t \lt(\lt(\frac{\sinh t}t\rt)^{n-1} -1\rt) dt = 1 + O(k^{-1}).
\]
Consequently, we get $C_k = (1 + O(k^{-1}))^{-\frac1n}$ and hence $C_k^{\frac n{n-1}} k = k + O(1)$. The estimates above shows that $\|u_k\|_{n,g}^n = O(k^{-1})$. Hence
\begin{align*}
\frac1{\|u_k\|_{n,g}^n}\int_{\B^n} \frac{\Phi_n(\al |u_k|^{\frac n{n-1}})}{(1+ |u_k|)^{p}} d\Vol_g& \geq \frac C{\|u_k\|_{n,g}^n}\int_{B_g(0,e^{-k})} \frac{\Phi_n(\al |u_k|^{\frac n{n-1}})}{(1+ |u_k|^{\frac n{n-1}})^{\frac{p(n-1)}n}} d\Vol_g\\
&\geq \frac{C k \Phi_n(\al\om_{n-1}^{-\frac1{n-1}} C_k^{\frac n{n-1}} k) }{ (1 + \om_{n-1}^{-\frac1{n-1}} C_k^{\frac n{n-1}} k)^{\frac {pn}{n-1}}} \int_0^{e^{-k}} (\sinh t)^{n-1} dt\\
&\geq C k^{1-\frac{p(n-1)}{n}} e^{-nk} \Phi_n(\frac{\al}{\al_n} n k + O(1))\\
&\geq C k^{1-\frac{p(n-1)}{n}} e^{nk(\frac{\al}{\al_n} -1)},
\end{align*}
here we use $C_k^{\frac n{n-1}} k = k + O(1)$, and $C$ denotes a constant which does not depend on $k$ and which value can be changed in each lines. Consequently, we get
\[
\lim_{k\to \infty} \frac1{\|u_k\|_{n,g}^n}\int_{\B^n} \frac{\Phi_n(\al |u_k|^{\frac n{n-1}})}{(1+ |u_k|)^{p}} d\Vol_g = \infty,
\]
if $\al > \al_n$ and for any $p$, or $\al = \al_n$ and for any $p < \frac n{n-1}$. This finishes the proof of Theorem \ref{Exact}.

\section*{Acknowledgments}
The author woul like to thank the anonymous referee for his/her useful comment and suggestion which improve the presentation of this paper.

\end{document}